\documentclass[12pt]{amsart}
\usepackage{amsmath,amscd}
\usepackage{geometry,amsfonts,amssymb,amsthm,txfonts,pxfonts,amscd,algorithmic,algorithm} 
\usepackage{refcheck,graphicx,url}
\norefnames
\nocitenames
\def\struckint{\mathop{%
\def\mathpalette##1##2{\mathchoice{##1\displaystyle##2}%
  {##1\textstyle##2}{##1\scriptstyle##2}{##1\scriptscriptstyle##2}}%
\mathpalette
{\vbox\bgroup\baselineskip0pt\lineskiplimit-1000pt\lineskip-1000pt
\halign\bgroup\hfill$}
{##$\hfill\cr{\intop}\cr\diagup\cr\egroup\egroup}%
}\limits}
\newtheorem{theorem}{Theorem}[section]

\newtheorem{corollary}[theorem]{Corollary}
\newtheorem{conjecture}[theorem]{Conjecture}
\newtheorem{definition}[theorem]{Definition}
\newtheorem{fact}[theorem]{Fact}

\theoremstyle{remark}
\newtheorem{observation}[theorem]{Observation}

\newtheorem{example}[theorem]{Example}

\newtheorem{question}[theorem]{Question}

\newcommand{\integers}{\mathbb{Z}}

\newcommand{\reals}{\mathbb{R}}

\DeclareMathOperator{\tr}{tr}

\DeclareMathOperator{\mcg}{Mod}

\DeclareMathOperator{\Sp}{Sp}
\DeclareMathOperator{\SL}{SL}

\DeclareMathOperator{\GL}{GL}

\begin{document}
\title{Generic phenomena in groups -- some answers and many questions}
\author{Igor Rivin}
\address{Department of Mathematics, Temple University, Philadelphia}
\email{rivin@temple.edu}
\date{\today}
\dedicatory{To the memory of Bill Thurston, with gratitude}
\begin{abstract}
We give a survey of some known results and of the many open questions in the study of generic phenomena in geometrically interesting groups.
\end{abstract}
\thanks{The author would like Ilan Vardi, Alex Eskin, Inna Capdeboscq, Peter Sarnak, Tania Smirnova-Nagnibeda and Tobias Hartnick for enlightening conversations, and the editors for their patience}
\maketitle
\tableofcontents
\section{Introduction}
In this paper we will discuss a number of loosely related questions, which emanate from Thurston's geometrization program for three-dimensional manifolds, and the general Thurston ``yoga'' that most everything is hyperbolic. We venture quite far afield from three-dimensional geometry and topology -- to the geometry of higher rank symmetric spaces, to number theory, and probability theory, and to the theory of finite groups. In Section \ref{thurston} we describe the underpinnings from the theory of three-dimensional manifolds as envisaged by W. Thurston. In Section \ref{idealist} we will describe one natural approach to describing randomness in groups. In Section \ref{seekrandom} we describe an approach to actually producing random matrices in lattices in semisimple Lie groups using the philosophy in Section \ref{idealist}. In Section \ref{nonidealist} we describe a different approach to randomness, and the questions it raises.
\section{Thurston geometry}
\label{thurston}
IAs far as this paper is concerned, history begins with Bill Thurston's geometrization program of three-dimensional manifolds. We will begin with the fibered version The setup is as follows: we have a surface $M$ (a  two-dimensional manifold, homeomorphic to a compact surface with a finite number of punctures) and a homeomorphism $\phi: M\rightarrow M.$ Given this information we construct the \emph{mapping torus} $T_\phi(M)$ of $\phi,$ by first constructing the product $\Pi = M \times [0, 1],$ and then defining $T_\phi(M)$ to be the quotient space of $\Pi$ by the equivalence relation which is trivial outside $M \times \{0, 1\},$ where $(x, 0) \sim (\phi(x), 1).$ One of Thurston's early achievements was the complete understanding of geometric structures on such fibered manifolds. To state the next results we will need to give a  very short introduction to  the \emph{mapping class group} $\mcg(M),$  which is the group of homeomorphisms of our surface modulo the normal subgroup of homeomorphism isotopic to the identity -- for a longer introduction, see the recently published (but already standard) reference \cite{farb2011primer}.  In low genus, the mapping class group is easy to understand. For $M \simeq \mathbb{S}^2,$ $|\mcg(M)| = 2;$ every automorphism of the sphere is isotopic to either the identity map or the antipodal map. The next easiest case is that of the torus: $M \simeq \mathbb{T}^2.$ Then, $\mcg(M) \simeq \GL(2, \mathbb{Z}).$ Looking at this case in more detail, we note that the elements of $\GL(2, \mathbb{Z})$ fall into three classes: elliptic (those with a fixed point in the upper halfplane), parabolic (those with a single fixed point $p/q$ on the real axis in $\mathbb{C}$) and the rest (these are hyperbolic, and have two quadratic irrational fixed points on the real axis). Elliptic elements are  \emph{periodic}. Parabolic elements leave the $(p, q)$ curve on the torus invariant (they correspond to a \emph{Dehn twist} about this curve). Hyperbolic elements leave no curve invariant. Further, one of their fixed points is attracting, while the other one is repelling. These two fixed points correspond to two orthogonal curves of \emph{irrational slope} on the torus.

In the case where $M$ is the torus with one puncture, Nielsen had proved that $\mcg{M}$ is the same as for $M\simeq \mathbb{T}^2.$ After that, things were rather mysterious, until Thurston discovered his classification of surface homeomorphisms, which parallels closely the toral characterization. Thurston's result is that every surface homemorphism falls into three classes: it is either periodic, or leaves invariant a \emph{multicurve} $\gamma$ (a collection of simple closed curves on $M$) -- in this case the map is allowed to permute the components of $\gamma,$ or \emph{pseudo-anosov}, in which case the map has a pair of orthogonal measured foliations, one of which is expanded by $\phi$ and the other is contracted. This is a highly non-trivial result which is the beginning of the modern two-dimensional geometry, topology, and dynamics. For a discussion in considerably more depth, see the standard references \cite{MR956596,MR1134426,MR964685,MR2850125}. The next theorem ties the above discussion into Thurston's geometrization program for 3-dimensional manifolds (the special case of fibered manifolds was probably the first case of geometrization finished -- see J.~P.~Otal's excellent exposition in \cite{MR1855976}. For an in-depth discussion of the various geometries of three-dimensional manifolds, see G.~P.~Scott's paper \cite{MR705527}.

\begin{theorem}[Thurston's geometrization theorem for fibered manifolds]
Let $T_\phi(M)$ be as above. Then we have the following possibilities for the geometry of $M.$
\begin{enumerate}
\item If $M \simeq \mathbb{S}^2,$ then $T_\phi(M)$ is modeled on $\mathbb{S}^2 \times \mathbb{R}.$
\item If $M \simeq \mathbb{T}^2,$ then we have the following possibilities:
\begin{enumerate}
\item If $\phi$ is elliptic, then $T_\phi(M)$ is modeled on $\mathbb{E}^3.$
\item If $\phi$ is parabolic, then $T_\phi(M)$ is a nil-manifold.
\item If $\phi$ is hyperbolic, then $T_\phi(M)$ is a solv-manifold.
\end{enumerate}
\item If $M$ is a hyperbolic surface, then
\begin{enumerate}
\item If $\phi$ is periodic, then $T_\phi(M)$ is modeled on $\mathbb{H}^2 \times \mathbb{R}.$
\item If $\phi$ is reducible, then $T_\phi(M)$ is a graph-manifold.
\item If $\phi$ is pseudo-Anosov, then $T_\phi(M)$ is hyperbolic.
\label{hypcase}
\end{enumerate}
\end{enumerate}
\end{theorem}

An attentive reader will note there are seven special cases, and six out of the eight three-manifold geometries make an appearance. Six out of the seven special cases of the theorem are easy, while the proof of the  last case \ref{hypcase} occupies most of the book \cite{MR1855976}. Thurston's philosophy, moreover, is that ``most'' fibered (or otherwise) three-manifolds are hyperbolic -- the first appearance of this phenomenon in Thurston's work is probably the Dehn Surgery Theorem (\cite{thurston1979geometry}), which states that moth Dehn fillings on a cusped hyperbolic manifold yield hyperbolic manifolds), and the last appears in his joint work with Nathan Dunfield \cite{MR2257389,MR1988291}, where it is conjectured that a random three manifold of fixed Heegard genus is hyperbolic. The actual statement that a random \emph{fibered} manifold is hyperbolic seems to have not been published by Thurston, and the honor of first publication of an equivalent question goes to Benson Farb in \cite{MR2264130}. We will discuss Farb's precise question below, but first, let's talk about what it means for some property $P$ to be generic for some (possibly) infinite (but countable) set $S.$ 
\section{An idealist approach to randomness} 
\label{idealist}
First, define a measure of size $v$ on the elements of $S.$ This should satisfy some simple axioms, such as:
\begin{enumerate}
\item $v(x) \geq 0$ for all $x \in S.$
\item The set $S_k = \{x\in S \left| v(x) \leq k\right.\}$ is finite for every $k.$
\end{enumerate}
Let now $P$ be a predicate on the elements of $S$ -- think of a predicate as just a function from $S$ to $\{0, 1\}.$ Let $\mathcal{P} \subset S$ be defined as $\mathcal{P} = \{x\in S\left| P(x) =1\right.\},$ and define $P_k = \{x\in \mathcal{P}\left| v(x) \leq k\right.\}.$ We say that the property $P$ is \emph{generic} for $S$ with respect to the valuation $v$ if 
\begin{equation}
\boxed{
\lim_{k\rightarrow \infty} \dfrac{|P_k|}{|S_k|} = 1.}
\end{equation}
We say that $P$ is \emph{negligible} with respect to $v$ if 
\begin{equation}
\boxed{\lim_{k\rightarrow \infty} \dfrac{|P_k|}{|S_k|}= 0.}
\end{equation}
Sometimes the above two definitions are not enough, and we say that $P$ has asymptotic density $p$ with respect to $v$ if 
\begin{equation}
\boxed{\lim_{k\rightarrow \infty} \dfrac{|P_k|}{|S_k|}= p.}
\end{equation}

These definitions work well when they work. Here are some examples:
\begin{example} The set $S$ is the set $\mathbb{N}$ of natural numbers, and the predicate $P$ is $P(x) = \mbox{is $x$ prime?}.$ The valuation $v$ is just the usual ``Archimedean'' valuation on $\mathbb{N},$ and, as is well-known, the set of primes is \emph{negligible}. One can make a more precise statement (which is the content of the Prime Number Theorem, see \cite{MR1790423,MR602825}):

With definitions as above, \[\dfrac{P_k}{S_k} = \Theta(\frac1{\log k}).\]
\end{example}
\begin{example}
\label{visible}
Let $S$ be the set of integer lattice points $(x, y) \in \mathbb{Z}^2,$ let $\Omega \subset \mathbb{R}^2$ be a Jordan domain, and define the valuation on $S$ as follows:
\[v(x) = \inf\{t \left| x \in t \Omega \right.\}.\] Further, define the predicate $P$ by $P(x, y) = \mbox{$x$ is relatively prime to $y.$}--$ such points are called \emph{visible}, since one can see them from the origin $(0, 0).$ Then, the asymptotic density of $P$ is $\frac1{\zeta(2)}=\frac6{\pi^2}.$ 

The proof of this for $\Omega$ being the unit square is classical, and can be found (for example) in Hardy and Wright (\cite{MR2445243}) or in the less classical reference \cite{MR1866856}. To get the general statement, we first note that the special linear group $\SL(2, \mathbb{Z})$ acts ergodically on the plane $\mathbb{R}^2$(see \cite{MR776417}). Now, define a measure $\mu_t$ by \[\mu_t(\Omega) =  \frac1{t^2}\mbox{the number of points such that $P(x, y) =1$ in $t\Omega.$}\] Each $\mu_t$ is clearly a measure, dominated by by the Lebesgue measure, and invariant under the $\SL(2, \mathbb{Z})$ action on $\mathbb{R}^2.$ By Helly's theorem \cite[Section 10.3]{MR1892228} It follows that the set $\{mu_t\}$ has a convergent subsequence $\sigma,$ and by $\SL(2, \mathbb{Z})$ invariance, the limit $\mu_\sigma$ is a constant multiple of the Lebesgue measure, and the constant can be evaluated for some specific $\Omega,$ such as the square (more details of the argument can be found in \cite{MR2318624}). Notice that the constant does not depend on $\sigma,$ so all the convergent subsequences of the set $\{\mu_t\}$ have the same limit, which must, therefore, be the unique limit point of the set.
\end{example}

\begin{example}
Consider the free group on two generators $F_2=\left<a, b\right>$ We define the valuation $v(x)$ to be the reduced word length of $x.$ Let $P$ be the predicate: $P(x) = \mbox{the abelianization $a(x) \in \mathbb{Z}^2$ is a visible point}.$ Then $P$ does \emph{not} have an asymptotic density. It does, however, have a an asymptotic \emph{annular} density, defined as follows: Let $X \subset S,$ where $S,$ as usual, has a valuation satisfying our axioms. We define $S_k = \{x \in S \left| v(x) = k\right.\},$ and similarly for $X_k.$ Then, the $k$-th annular density of $T$ is defined by 
\begin{equation}
\rho_k(X) = \dfrac12 \left(\dfrac{X_{k-1}}{S_{k-1}} + \dfrac{X_k}{S_k}\right).
\end{equation}
We define the \emph{strict annular density} of $X$ to be $\rho_A(X) = \lim_{k\rightarrow \infty} \rho_k(X),$ if the limit exists.
The general result (shown in \cite{MR2318624}) is:
\begin{theorem}
\label{annular}
Let $S$ be an $\SL(n, \mathbb{Z})$ invariant subset of $\mathbb{Z}^k,$ and let $\tilde{S} = a^{-1}S,$ where $a,$ as before, is the abelianization map from the free group on $k$ generators $F_k$ to $\mathbb{Z}^k.$ Then $\tilde{S}$ has a strict annular density whenver $S$ has an asymptotic density. Moreover, the two densities are equal.
\end{theorem}
The proof of Theorem \ref{annular} uses the ergodicity of the $\SL(n, \mathbb{Z})$ action on $\mathbb{R}^n,$ as described in Example \ref{visible}, and the central and local limit theorems of \cite{rivin1999growth} (see also \cite{rivin2010growth}) and of \cite{sharp2001local}.
\end{example}

The examples above show that the cases where the groups are reasonably simple to describe, the idealistic valuation-based approach is quite successful. However, once the groups are more complicated, this approach often bogs down in at least some ways, the principal of which is that when one talks of negligibility, genericity, or density, one is making a statement about properties of \emph{random} elements of the set $S.$ However, this raises the question of how to generate such random elements (note that the generation method will often hold the keys to our ability to approach asymptotic statements.

\begin{example} Let $S = \SL(n, Z).$ There is a natural family of valuations on $S$ -- the archimedean valuations associated to the various Banach space norms on the space of matrices$M{n\times n}.$ Since all these are known to be equivalent, we might as well choose the frobenius norm (the $L^2$ norm of a matrix $x\in \SL(n, \mathbb{Z})$ viewed as a vector in $\mathbb{Z}^{n\times n}.$). In other words, in our previous language, \[
v(x) =\sqrt{\sum_{i=1}^n \sum_{j=1}^n a_{ij}^2}.
\]
Let $S_{\leq k}$ be the set of those $x$ in $S$ with $v(x) \leq k.$ It is not at all obvious how to find the cardinality of $S_{\leq k},$ though this has been done (relatively recently) for $\SL(2, \mathbb{Z})$ (by Morris Newman in \cite{MR924457} -- see more on this in Section \ref{sl2z1}) and in general by W. Duke, Z. Rudnick, and P. Sarnak in \cite{duke1993density}, A. Eskin and C. McMullen in \cite{eskin1993mixing}.  The result is that the number of points is asymptotic to a constant times $k^{n^2-n}$ -- the constant for $n=2$ is 6. In any case, enumeration in and of itself is difficult, and enumerating subsets seems more difficult still. An example of this is the (simple to state) question of finding a \emph{uniformly distributed random element} of bounded norm (see section \ref{seekrandom} for more).
\end{example}

Nevertheless, one can try to show some results on linear groups using the the Archimedean valuation as above, sometimes using the very deep results of P. Sarnak and A. Nevo as \cite{nevo2010prime}, which are a major advance on the results of \cite{duke1993density}.
\begin{example} We are using \cite{rivin2008walks,rivin2009walks} I show that a generic element of $\SL(n, \mathbb{Z})$ has irreducible characteristic polynomial; furthermore, the Galois group of the characteristic polynomial is generically the full symmetric group. In the same references I show that the generic element of $\Sp(n, \mathbb{Z})$ has irreducible characteristic polynomial. 
\end{example}

\begin{example} Let $G$ be a lattice in a semi-simple linear group $\mathfrak{G},$ pick a (non-central) element $h\in G$ and for an element $x \in G$ consider the group $H_x = \left< x, h\right>.$ Consider the predicate $P(x) = \mbox{$H_x$ is Zariski dense in $\mathfrak{G}$}.$ In the paper \cite{MR2725508} I show that $P$ is generic in $G.$ Similarly, if we $H = G\times G,$ then the property $P(x, y) = \mbox{$H_x$ is Zariski dense in $\mathfrak{G}$}$ is generic in $G\times G.$The results are based on strong approximation theory, as developed in \cite{MR735226}; see also \cite{MR1278263}.
\end{example}
\begin{question}
\label{zdenseq}
How do we tell if a subgroup $G$ of (say) $\SL(n, \integers)$ given by its matrix generators is Zariski dense?
\end{question}
There are two different avenues by which to attack Question \ref{zdenseq}. The first is via strong approximation techniques: it is known (by \cite{MR735226}) that almost all modular projections are surjective for a Zariski dense subgroup $G,$ and further, by the results of Weigel \cite{MR1600994}, the group is Zariski dense if \emph{any} projection modulo some $p > 3$ is surjective. So, we need only check a finite number of possible bad projections, which \emph{is} bounded by the work of Rapinchuk \cite{rapinchuk2012strong}, but the bound is not what one would call practical, so this approach, while aesthetically pleasing, takes a lot of work to make work.

A completely different approach is the brute force attack: take the group, compute several elements, compute their (matrix) logarithm, and see if the resulting elements generate the Lie algebra of $\SL(n, \mathbb{C})$ as a vector space. This is a much more computationally promising approach, but it requires a lot of work to produce provable results (the logarithm can usually be computed only approximately, so one needs to find the measure of one's confidence in one's results, etc).  The method is particularly effective when there are a lot of unipotent elements in the subgroup, since the logarithm of a unipotent integral matrix is a matrix with \emph{rational} entries, so no approximation techniques are necessary (this was pointed out by to the author by A. Eskin).

\begin{example}
\label{fuchses}
 Jointly with Elena Fuchs \cite{fuchsrivin} we show if $G$ is a lattice in $\SL(2, \mathbb{C}),$ and $H = G \times G,$ then, for any $\alpha > 0,$ the property 
\[P(x, y) = \mbox{The Hausdorff dimension of $\left< x, y\right>$ is at least $\alpha$}\] is negligible in $H$. 

The argument uses the ergodicity of the action of $\SL(2, \mathbb{Z})$ on the plane (which shows that the attractive and repelling fixed points of elements are equidistributed), and a ping-pong argument, together with bounds on the Hausdorff dimension as in \cite{mcmullen1998hausdorff}.
\end{example}

\begin{question} 
\label{myrberg} A closely related question to the one discussed in Example \ref{fuchses} is open: Consider a pair of elements \emph{one of which is parabolic}. Is it true that the property 
\[P(x, y) = \mbox{The Hausdorff dimension of $\left< x, y\right>$ is at least $\alpha$}\] is negligible for any $\alpha > 1/2?$
\end{question}

It should be remarked that it is a theorem of Beardon \cite{beardon1968exponent} that any such group \emph{does} have Hausdorff dimension at least $1/2,$ so the constant $1/2$ in Question \ref{myrberg} is best possible.

It is hoped that the techniques used to attack Example \ref{fuchses} can be extended to attack Question \ref{myrberg}.
\begin{example} 
\label{innaex}
Jointly with Inna Capdeboscq \cite{innarivin} we show that if $G$ is a lattice in a semi-simple Lie group $\mathfrak{G}$ of rank at least two, and $H=G\times G,$ then if we define $P(x, y) = \mbox{$\left<x, y\right> $ is profinitely dense},$ then $P(x, y)$ has asymptotic density bounded below. 

The idea of the argument is as follows:

First observation is that $SL(n, \mathbb{Z}/NM\mathbb{Z}) = SL(n, \mathbb{Z}/N\mathbb{Z}) \times SL(n, \mathbb{Z}/M\mathbb{Z}),$ for $N, M$ relatively prime.

Second observation is that random elements (either in the random walk model or in the "archimedean height" model) are eventually equidistributed in modular projections, e.g modulo the product $P$ of the first $k$ primes (this is one of the results of \cite{rivin2008walks}). By the first observation, the behaviors modulo different primes are independent,  and so by \cite{MR1065213,MR1338320} the probability that the projections onto the first $k$ primes are surjective is bounded below by \[B_k = \prod_{i=1}^k (1-C(n)/p^{n-1}),\] where $C(n)$ is their rank-dependent constant. Since the series \[\sum_{i=1}^\infty \frac{1}{p^{n-1}}\] converges for $n>2,$ it follows that the products $B_k$ converge to some constant $B.$ 

Now, for any $\epsilon$ we can pick $k$ in such a way that $|B_k- B |<\epsilon,$ while \[R_k=\sum_{i=k+1}^\infty \frac{C(n)}{p^{n-1}} < \epsilon.\] By the union bound, the probability that some projection is not surjective is bounded above by $(1-B_k) + R_k  \leq B + 2\epsilon.$ So, as long as $\epsilon \ll (1-B)/2,$ we get the probability that at least one projection does not surject is bounded above by $(1-B)/2.$ In reality, of course, if the walks get very very long, the true probability of profinite density is bounded below by $B.$

The observations above show that the modular projections are surjective for all \emph{prime} moduli with positive probability.  Then, using some group-theoretic arguments we can show that there is a positive probability of surjection for all moduli. Then probabilities one gets are completely effective. 
\end{example}
Since the first example of a profinitely dense (free) subgroup of $\SL(n, \integers)$ was constructed by Steve Humphries in his beautiful paper \cite{MR944152} , we call the groups described in Example \ref{innaex} \emph{Humphries groups}. Now for the question:

\begin{question}
\label{humphriesq}
Suppose we are given a subgroup $G$ of $\SL(n, \integers)$ given by matrix generators. How hard it is to decide whether it is a Humphries group?
\end{question}
Question \ref{humphriesq} is quite difficult. Notice that the full lattice $\SL(n, \integers)$ is a Humphries group (by the definition above), so there is a natural dichotomy:  Either $G$ is the whole $\SL(n, \integers)$ or, by the congruence subgroup property, it is infinite index, and further, by \cite{MR1732043} (or the sharper and more general \cite{MR1600994}), $G$ is Zariski-dense in $\SL(n, \integers).$ This dichotomy is not really relevant from the algorithmic standpoint -- we know from the work of Matthews, Vaserstein, Weisfeiler \cite{MR735226} that for any Zariski-dense subgroup $G$ only a finite number of projections is not surjective, and if we could bound the number, we could just check surjectivity for every possible exceptional modulus -- such a check can be performed efficiently using, for example, the algorithm of Neumann and Praeger \cite{MR1182102}. Unfortunately, getting a bound using the generators is not so easy. The first advance came (after the author raised the question at an MSRI Hot Topics conference) very recently, in the work by A.~Rapinchuk \cite{rapinchuk2012strong}, but the bounds there, while explicit, are not really computationally useful, as Rapinchuk prominently states in the paper.

On the toher hand, the beginning of the discussion in the paragraph above begs the question:
\begin{question}
\label{slnq}
Given a collection of matrices in $\SL(n, \integers)$ do they generate $\SL(n, \integers)?$
\end{question}

There appears to be only one practical approach: that is, compute the fundamental domain of the span of the matrices on the homogeneous space of $\SL(n, \mathbb{R}).$ If we are lucky, and that domain is finite-volume, we can answer the question (the author has conducted a number of experiments along these lines). No general attack seems to be available, and it is not even clear whether the question is decidable! Similar sounding questions (like the membership problem) are undecidable in $\SL(n, \integers)$ for $n \geq 4,$ but the techniques seem unapplicable here. In special cases (which are central to the study of mirror symmetry, see \cite{singh2012arithmeticity,brav2012thin}) the question can be decided by a rather diverse set of approaches.

\section{Punctured (or not) torus}
\label{sl2z1}
Consider the modular group $\SL(2, \mathbb{Z}).$ Our first set of results will use ordering by Frobenius norm of the matrix. 
\begin{definition}
The frobenius norm of the matrix $x=\begin{pmatrix}a & b\\c & d\end{pmatrix}$ is $\|x\| = \sqrt{a^2 + b^2 + c^2 + d^2}.$ 
\end{definition}
The first question is:
\begin{question}
How many elements $x\in \SL(2, \mathbb{Z})$ have $\|x\| \leq N?$
\end{question}
It is surprising that this question was first answered by Morris Newman in 1988(!) \cite{MR924457}:
\begin{theorem}
\label{newmanthm}
The number $\mathcal{N}_k$ of elements $x \in \SL(2, \mathbb{Z})$ with $\|x\| \leq k$ is asymptotic to $6 k^2.$
\end{theorem}
M. Newman's proof of Theorem \ref{newmanthm} begins by reparametrizing $\SL(2),$ as follows.
First define the following variables:
\begin{gather*}
A = a + d \\
B=b+c\\
C=b-c\\
D=a-d
\end{gather*}
We see that $A^2+B^2 + C^2 + D^2 = a^2 + b^2 + c^2 + d^2.$ Further note that 
\begin{equation}
\label{quadeq}4 = 4 (ad - b c) = A^2 +C^2 - B^2- D^2,
\end{equation}
while
\begin{equation}
\label{lineq}A = \tr \begin{pmatrix} a & b \\ c& d\end{pmatrix}.
\end{equation} Since the difference between $A$ and $D$ is $2 d,$ we know that
\begin{equation}
\label{cong1}
 A \equiv D \mod 2,
\end{equation}
 and for the same reason 
\begin{equation}
\label{cong2}
B\equiv C \mod 2.
\end{equation}
Then Newman writes down a generating function for the number of matrices in $\SL(2, \integers)$ with prescribed Frobenius norm in terms of the theta function 
\[\theta(x) = \sum_{n=-\infty}{\infty} x^{n^2},\] and uses classical estimates on the coefficients of products of theta functions to obtain the asymptotic result of Theorem \ref{newmanthm}. Since an exposition of this method would take us too far afield,let's use the parametrization above to  count those elements with trace equal to $2$ (which is to say, the parabolic elements). Equations \eqref{quadeq} and \eqref{lineq} tell us that the number of such matrices with Frobenius norm bounded by $k$ is exactly equal to the number of \emph{Pythagorean triples} of norm bounded by $k.$

Now, as is well-known, pythagorean triples $(A, B, C)$ with $A^2 = B^2 + C^2$ are rationally parametrized by 
\begin{gather}
A = u^2 + v^2\\
B = u^2 - v^2\\
C= 2 u v,
\end{gather} With this parametrization, the $2$-norm of $(A, B, C)$ equals $\sqrt{2} (u^2 + v^2),$ so the number of pythagorean triples with $L^2$ norm bounded above by $X$ equals the number of pairs $(u, v)$ with $L^2$ norm bounded above by $2^{1/4} \sqrt{X},$ which, in turn, is asymptotic to $\sqrt{2} \pi X.$  Note that the congruences \eqref{cong1} and \eqref{cong2} tell us that $2 u v = a- d.$ This overcounts by a factor of two (since $(u, v)$ and $(-u, -v)$ give the same Pythagorean triple, but on the other hand, parabolic matrices are allowed to have trace equal to $\pm 2,$ so when the smoke clears, we have:
\begin{theorem}
The number of parabolic matrices in $\SL(2, \mathbb{Z})$ and Frobenius norm bounded above by $k$ is asymptotic to $\sqrt{2} \pi k.$
\end{theorem}

Now, we make the following observation: 
\begin{observation}
The characteristic polynomial of a matrix in $\SL(2, \integers)$ factors over $\integers$ if and only if the matrix is parabolic (so has trace $\pm 2.$)
\end{observation}
\begin{proof}
Indeed, if $M \in \SL(2, \integers),$ the roots of the characteristic polynomial $\chi(M)$ are
\[
\dfrac{\tr M \pm \sqrt{\tr^2 M - 4}}2.
\]
For $\chi(M)$ to factor, $\tr^2 M - 4$ must be a perfect square, which obviously happens only when $|\tr M| = 2.$
\end{proof}
We thus have the following:
\begin{theorem}
The probability of a matrix in $\SL(2, \integers)$ of Frobenius norm bounded above by $x$ to have \emph{reducible} characteristic polynomial is asymptotic to $\frac{3 \sqrt{2}}{\pi x}.$
\end{theorem}

\section{Looking for random integer matrices}
\label{seekrandom}
Consider the following simple question: Let $v$ be an Archimedean valuation in a lattice (below we will be discussing $\SL(n, Z),$ but the question is just as interesting for any other lattice in a non-compact Lie Group (not necessarily semi-simple).
\begin{question}
\label{randsl}
Given $k,$ how do we choose a random element $x$ \emph{uniformly} with $v(x) \leq k?$
\end{question}
Even for $\SL(2, \mathbb{Z}),$ Question \ref{randsl} seems completely open, but here is an idea:

\subsection{A line of attack for $\SL(2, \integers)$}
\label{sl2rand}
To get an \emph{approximately} uniform element, based on the fact that the homogeneous space of $\SL(2, \mathbb{R})$ is the hyperbolic plane $\mathbb{H}^2:$ Take a basepoint in $\mathbb{H}^2$ (since we will be eventually interested in $\SL(2, \mathbb{Z}),$ the Poincar\'e halfspace model is popular, and there the point $i = \sqrt{-1}$ is a popular choice of basepoint). Now, the matrices in $\SL(2, \reals)$ with  Frobenius norm  bounded by $N$ translate $i$ by hyperbolic distance at most some $f(N),$ so pick a disk $D$ of radius $g(N)$ in $\mathbb{H}^2,$ and pick a point $x$ uniformly at random. $x$ will lie in some fundamental domain of the $\SL(2, \mathbb{Z})$ action. The point $x$ corresponds to a lattice $L \subset \reals^2,$ which can be \emph{reduced} (using Legendre' algorithm -- basically continued fractions) -- this corresponds to finding a matrix $m(x)$ in $\SL(2, \mathbb{Z})$ which maps the fundamental domain of $x$ to the ``standard'' fundamental domain of the modular group $\SL(2, \mathbb{Z}).$ The matrix $m(x)$ is our candidate for the uniformly random element of $\SL(2, \mathbb{Z})$ we seek (the fact that equidistribution on $\mathbb{H}^2$ leads to equidistribution on $\SL(2, \reals)$ is standard, see for example \cite{eskin1993mixing}).

There are two problems with the above approach. Firstly, $m(x)$ might not satisfy the norm constraint. If that is the case, we throw it away, and try again -- if $g(N)$ is not too big, this process will terminate reasonably quickly. The other problem is that the area in the disk $D$ is only approximately equidistributed amongst fundamental domains (more precisely, the intersection of $D$ with the union of the translates of the basic fundamental domain by matrices satisfying the norm inequality is only approximately equidistributed among these translates). There is a trade-off between the two problems: the larger disk we take, the more uniform the distribution is, but the less likely we are to get a point satisfying the norm inequality, so some care is required in designing this algorithm properly.

\subsection{A line of attack for $\SL(n, \integers)$ and $\Sp(2n, \integers)$}
The method described in Section \ref{sl2rand} can be extended to higher dimensions, for at least the special linear and symplectic groups. The homogeneous space for $\SL(n, \integers)$ (symmetric positive definite matrices with determinant $1,$ and for $\Sp(2n, \integers)$ (the Siegel halfspace) have been known for several decades, and sampling uniformly from the ball in that space is easy, using what Lie Theorists call the $KAK$ decomposition, and most other people call the singular value decomposition. The Haar measure (induced by that on the Lie group) is easy to compute, and a random element in the ball of the homogeneous space is easy to sample. 

What is \emph{not} so easy is the lattice reduction step. Lattice reduction is a much studied problem, since the groundbreaking work of Lovasz (as embodied in the LLL) algorithm, a good survey is \cite{nguyen2011lattice} (interestingly, the symplectic version of the problem had not been considered until quite recently, see \cite{gama2006symplectic}). The problem is that the complexity of \emph{exact} lattice reduction is, at present, exponential (in dimension $n$ of the lattice. The LLL algorithms, and the various improvements run in polynomial time, but they don't necessarily get us to the canonical fundamental domain. They \emph{do} get us near the canonical fundamental domain, which brings up a fundamental question:
\begin{question}
What are the statistical properties of the currently used lattice reduction algorithms.
\end{question}
In other words, if we run the algorithm we sketched not with a precise lattice reducer, but with an approximate one (like LLL), will the matrices we get be uniformly distributed in the ball in $\SL(n, \mathbb{Z})$ or $\Sp(2n, \mathbb{Z})?$

The basic principle of the method described works for lattices over number field, and not just over $\integers.$ In fact, a version of lattice reduction for such is described in \cite{fieker2010short}. The version over $\SL(2, \integers)$ can be easily made to work to generate a random matrix in an \emph{arithmetic} Kleinian group -- the continued fraction algorithm analogue is described in \cite{page2012computing}. It would be interesting to analyse the \emph{non}-arithmetic case.

\section{A non-idealist approach to randomness}
\label{nonidealist}
A much more tractable, from the computational standpoint, approach to generating random elements of fairly arbitrary (finitely generated) groups is the following:

Take a symmetric generating set $S = \{g_1, \dotsc, g_k\}$ of our favorite group $G,$ (where ``symmetric'' means that the set is invariant under the map $x \mapsto x^{-1},$ and look at the set $W_k$ words of length  $k$ in the elements of $S.$ The statement that the group $G$ is finitely generated means that 
\[
\bigcup_{k\geq 0} W_k = G.
\]
Now, the trick is to use the definitions in Section \ref{idealist}, but apply them not to the group $G$ itself but to the \emph{free monoid}  $M_S$ on $S,$ with \[v(x) = \mbox{the word length of $x.$}\] This approach has many advantages: it is trivial to produce a random element (just multiply elements at random), there are different techniques for proving results, and, in the linear group context this model is closely related to the study of random matrix products, which has a long history and a considerable record of success (a standard reference is \cite{bougerol1985products}). The (fairly obvious) disadvantage is that the structure of $M_S$ has nothing to do with the structure of $G,$ and it is very difficult to estimate the relationship between the number of occurences of a given group element in a $v$-ball of some given radius $R.$ Consequently, doing probability on $M_S$ instead of $G$ is has a certain air of capitulation to it. On the other hand, the free monoid model can be refined, as follows: Elements in the free monoid can be identified with walks in the \emph{complete graph} $\mathcal{K}_k$ on $k=|S|$ vertices (we use the term ``complete graph'' in a somewhat nonstandard way: every vertex of ${K}_k$ is connected \emph{to itself} in addition to all the other vertices (see Figure \ref{freemon}).
\begin{figure}[hb]
\label{freemon}
\centering
\includegraphics[width=2in]{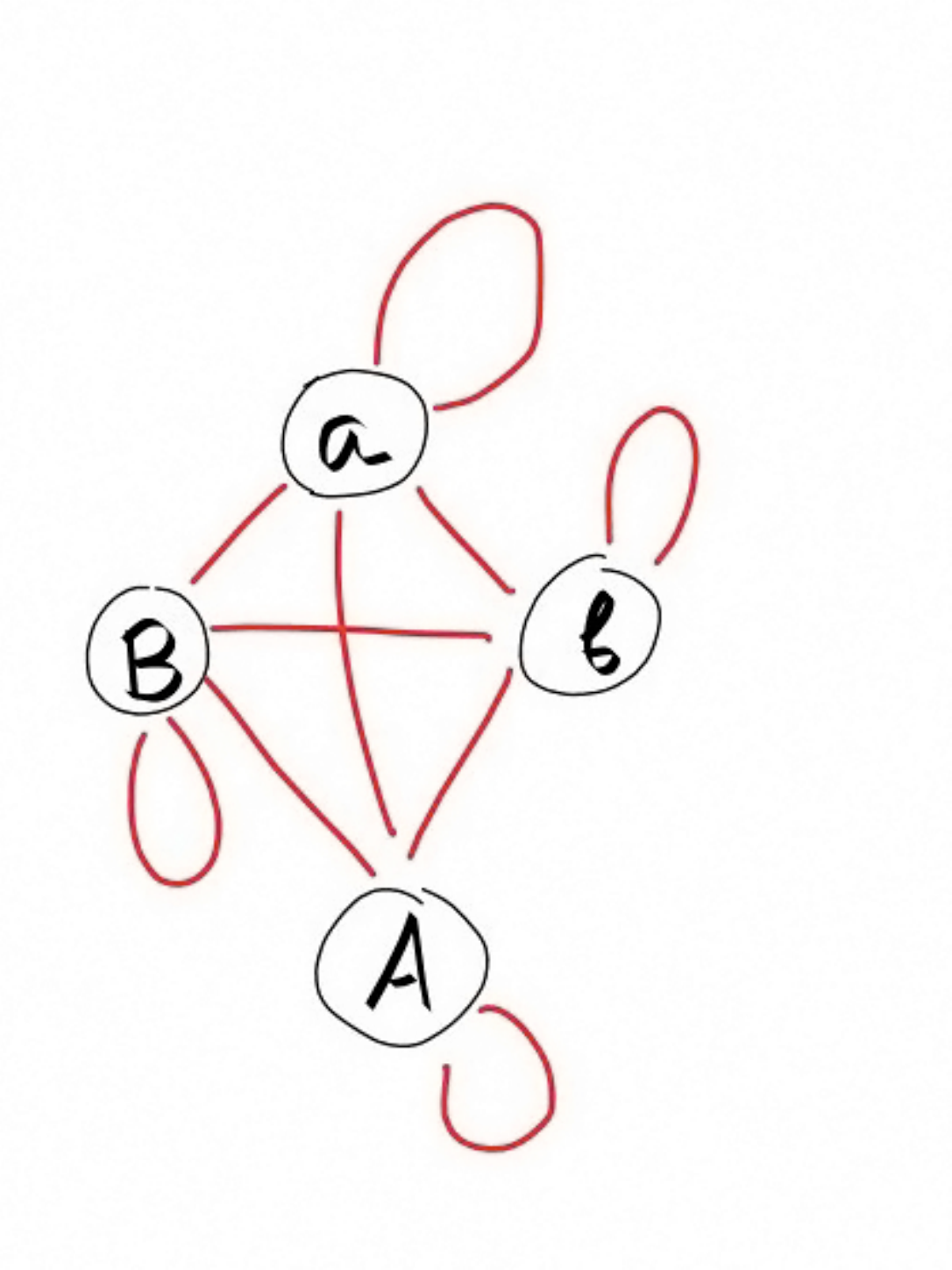}
\caption{The recognizing automaton for the free monoid on two generators  is the very complete graph; $bB=aA=1$}
\end{figure}
If we remove the requirement that a graph be the (very) complete graph, and, indeed, a directed, as opposed to undirected graph, we find ourselves in the world of \emph{regular languages}, and it was a major discovery of Jim Cannon's, expanded upon by a number of people, including David Epstein and Bill Thurston (see the classic book \cite{cannon1992word}) that such regular languages are a good way of describing a large class of groups (the so-called \emph{automatic groups}). For such groups, the length of a walk in the defining automaton is a very good valuation -- in particular,  it coinsides with the distance in the Cayley graph from the identity element).

However, it turns out that this is too broad a context to be able to demonstate sharp results, and so much of the author's work (see \cite{rivin2008walks,rivin2009walks,MR2725508} so far has centered on a smaller set of automatic structures: namely, we consider only \emph{undirected} graphs, which, in addition, have the \emph{Perron-Frobenius} property: there is a unique eigenvalue of the adjacency matrix of maximal modulus. This has the immediate benefit of bringing the \emph{free monoid} model closer to reproducing structures of actual interest. For example, here is the graph which generates reduced words (for two generator groups, the general case is similar):
\begin{figure}[hb]
\label{freegrp}
\centering
\includegraphics[width=2in]{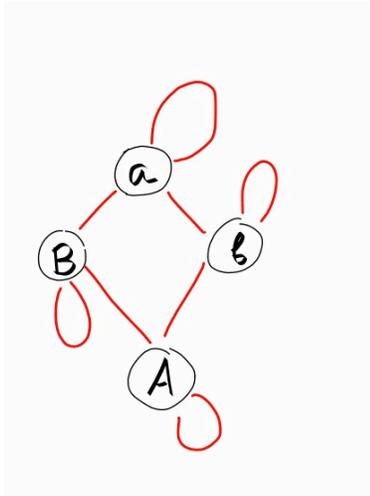}
\caption{The recognizing automaton for the free group on two generators $F_2.$; $bB=aA=1$}
\end{figure}
This brings up the relatively obvious question:
\begin{question}
\label{undirgroup}
Which groups have an automatic structure where the accepting automaton has the properties of being
\begin{enumerate}
\item undirected.
\item Perron-Frobenius?
\end{enumerate}
Let's call the combination of properties 1 and 2 \emph{property R}.
\end{question}
First, a conjecture:
\begin{conjecture}
Every word hyperbolic group has property R, with respect to some generating set.
\end{conjecture}
It is not clear that property R is generating set invariant.
While the answer to Question \ref{undirgroup} is obvious of interest, and the only groups known to have property R are free groups, there is a ``cheap'' way to extend the techniques to a bigger class of groups, as described in the next section.
\subsection{What if your group is not free?} As a simple example of a non-free group, we take the \emph{modular group} $\mathcal{M}=\SL(2, \integers).$ As is well known, this group is almost, but not quite free.  More precisely, 
\[\mathcal{M} = \langle S, T \left| S^2, T^3\right.\rangle = C_2 \star C_2,\] where $C_p$ is the cyclic group of order $p$ and $\star$ denotes the free product. The obvious symmetric automaton which accepts $C_p \star C_q$ has $p+q -2$ vertices, corresponding to $T, T^2, \dotsc, T^{p-1}, S, \dotsc, S^{q-1}.$ Every vertex corresponding to $T^i$ is connected to all vertices of the form $S^j.$ This works wonderfully, except for the minor matter of not representing the identity element and the not-so-minor matter of being bipartite, hence not having property R. However, this can seemingly be fixed by making a new graph with $(p-1)(q-1)$ vertices (corresponding to the products $S^i T^j$) and $q-1$ start states (corresponding to $T, \dotsc, q-1.$ I believe that this technique will allow the methods used in \cite{rivin2008walks,rivin2009walks,MR2725508} to be extended to this class of groups.
\section{A non-idealistic approach to $\SL(2, \integers)$}
Using the automata described in either of the Figures \ref{freemon} or \ref{freegrp}, we can study $\SL(2, \integers)$ using the random walk approach described in Section \ref{nonidealist}. The idea is simple: consider an $n$-step walk on the recognizing graph $G.$  Since the group $\SL(2, \integers)$ is a bit too big for us (it is infinite, for one thing), let's do a quick warm-up, and consider the group $\mathcal{M}_p=\SL(2, \integers/ p \integers).$ How many elements of $\mathcal{M}_p$ have trace equal to $2?$  A matrix $M \in \mathcal{M_p}$ has trace $2$ if it has the form $M = \begin{pmatrix} a & b \\ c & 2 - a\end{pmatrix},$ where $a (2 - a) - b c  = 1.$ The last equation can be rewritten as 
\begin{equation}
\label{paraeq}
bc + (a-1)^2 = 0.
\end{equation}
 Now, if $a=1,$ Eq.\eqref{paraeq} has $2p-1$ solutions ($p-1$ solutions with $b,$ but not $c$ equal to $0,$ $p-1$ solutions with $c,$ but not $b,$ equal to zero, and $(0, 0).$) 
If $a\neq 1,$ Eq. \eqref{paraeq} has $p-1$ solutions of the form $(b, -(a-1)^2/b).$ 
This gives a total of 
$2p-1 + (p-1)^2 = p^2$ matrices with trace equal to $2.$ 
On the other hand, the order of $\mathcal{M}_p$ is equal to  $p(p^2-1),$ so for large $p$ there is a probability of around $1/p$ that $M$ picked uniformly from $\mathcal{M}_p$ picked uniformly at random has trace  equal to $2.$ 

What does this have to do with the problem at hand? Note that a matrix which has trace equal to $2$ has trace equal to $2$ for \emph{every} prime $p.$ This means that if the walks on our graph $G$ are equidistributed in $\mathcal{M}_p$ for some $p,$ the asymptotic probability that an element is parabolic is \emph{at most} $1/p.$ The key is now the following theorem of \cite{rivin2008walks}:
\begin{theorem}
\label{finitethm}
Let $G$ be a graph with property $R,$ with vertices labeled by generators $\gamma_1, \dotsc, \gamma_v$ of a finite group $\Gamma.$ Then,  the walks of length $k$ become equidistributed in $\Gamma,$ exponentially quickly as a function of$k,$ \emph{unless} all of the $\gamma_i$ are sent to the same complex number by some one-dimensional irreducible representation of $\Gamma.$
\end{theorem}
Since $\mathcal{M_p}$ has no irreducible one-dimensional representations, we see that the asymptotic probability is smaller than any $1/p,$ and hence the asymptotic density is $0.$ In fact, using the fact that $\SL(2, \integers),$ while not enjoying property $T,$ does enjoy property $\tau$ for congruence representations (see \cite{MR2147485}), we can show that the probability of being parabolic decreases $\emph{exponentially}$ in the length of the walk (see \cite{rivin2009walks}).

It should be remarked that for specific generating sets, explicit growth rates have been computed in \cite{MR2721969,MR1845594}. In the (chronologically) the first of these papers (\cite{MR1845594}) Takasawa views $\SL(2, \integers)$ as the mapping class group of the torus. In the second, Korkmaz and Atalan study the mapping class group of the four-punctured sphere, but the two objects are, in fact, the same (though the generating sets are different). This follows from the fact that for any hyperbolic structure on the punctured torus there is the \emph{elliptic involution}, the quotient by which is an orbifold of signature $(0; 2, 2, 2, \infty),$ while each quadruply punctured sphere admits an order four symmetry group (the Klein four-group) of involutions (this can be seen in many ways, one of which being that each complete finite-area structure on the four times punctured sphere can be realized uniquely as the induced metric on an ideal symplex in $\mathbb{H}^3$, see \cite{MR1280952}), the quotient by which is the self-same orbifold. The reader wishing a much harder algebraic proof of this fact can consult \cite{MR1315914}.
\subsection{Polynomial versus exponential} The attentive and inquisitive reader may have noticed that in Section \ref{sl2z1} we showed that for the matrices of norm bounded above by $x,$ there was a probability of order $1/x$ of finding a parabolic, while this probability decreases \emph{exponentially} fast for the graph walk model. The simplest way of explaining this is the following: the number of distinct elements of length $k$ grows exponentially in $k$ (the order of growth depends on the generating set), so the probability of being parabolic is decaying only polynomially in the size of the sample space.

The second simplest explanation (closely related to the first) is that at least in the simplest possible walk model, the expected norm of the products grows exponentially in the length of the walk (this follows from the classical theory of random matrix products, see \cite{bougerol1985products}. 

The case of $\SL(2, \integers)$ is particularly interesting (as it always is). A particularly popular generating set for $\SL(2, \integers)$ is the set $\{L, U\}$ where \[L=\begin{pmatrix}1 & 0\\1 & 1\end{pmatrix}, \quad  U =\begin{pmatrix}1 & 1\\ 0 & 1\end{pmatrix}.\]
Any matrix$ M = \begin{pmatrix}a & b\\c & d\end{pmatrix},$ where we assume that $b>a>0, d > c \geq 0$ (the other cases are similar) can be written as $M = U^{a_0} L^{a_1} \dots U^{a_r},$ where $b/d = [a_0, a_1, \dotsc, a_r].$ In this case, the  word length of $M$ in terms of the generators $L, U$ is simply the sum of the continued fraction coefficients of $b/d$ -- the word length with respect to other generating sets is within a multiple (obtained by writing the ``new'' generators $S, T$ as words in $U, L$ and vice versa). It turns out that the sum of continued fraction coefficients is not so easy to analyze, though not for lack of trying. The first reference seems to be the paper of Andy Yao and Don Knuth in 1975 \cite{yao1975analysis}, where the authors show that the \emph{average} value of the sum $S(p/q)$ of the continued fraction coefficients of $p/q$ taken over all $0 < p < q$  satisfies:
\[\frac1q\sum_{0 < p < q} S(p, q) \asymp \frac{6}{\pi^2} (\log q)^2.\]
The $(\log q)^2$ growth is a little deceptive, since the distribution has ``fat tails.'' Indeed, I.~Vardi had shown in \cite{vardi1993dedekind} that for any $\alpha$ satisfying $1 > \alpha > \frac12,$ we have
\[
S(p/q) \leq (\log q)^{1+\alpha}, \quad p < q < n, \quad \mbox{with at most $O_\epsilon(N^2(\log N)^{1/2-\alpha + \epsilon})$ exceptions.}
\]
Vardi's paper is mostly concerned with Dedekind sums, which can be interpreted as the \emph{alternating} sum of continued fraction coefficients. The reader is referred to Vardi's very nice preprint \cite{vardi2009continued} for more on continued fractions and related mathematics.

\section{Higher mapping class groups}
\label{higher}
The techniques which work for the torus at first appear to fail resoundingly for the mapping class group of more complicated surfaces, since the mapping class group in that setting is not linear. Luckily, there is a workaround: The \emph{Torelli homomorphism}  $\mathcal{T}$ s a homomorphism from the mapping class group $\mathcal{M}(S)$ of a (closed, for simplicity, and oriented) surface $S$ of genus $g$ to $\Sp(2g, \integers).$ A mapping class $\phi$ is mapped by $\mathcal{T}$ to its action on the first homology group of $S.$  This action preserves the intersection pairing, and so the image of $\mathcal{T}$ is contained in the symplectic group. In addition, the following fact is standard:
\begin{fact}[see \cite{MR2850125}]
\label{torellionto}
The image of $\mathcal{T}$ is all of the symplectic group $\Sp(2g, \integers).$
\end{fact}
Except for the cases of the torus and the four-times punctured sphere, the map $\mathcal{T}$ has a nontrivial kernel, known as the \emph{Torelli group}  $\mathfrak{T}(S)$ of the surface $S.$ The Torelli group contains pseudo-anosov elements, so the Torelli homomorphism does lose a lot of information. Nonetheless, there is the following theorem due to A.~Casson (see \cite{MR964685}):
\begin{theorem}
\label{cassonthm}
Suppose the matrix $M=\mathcal{T}(\phi)$ has the following properties:
\begin{enumerate}
\item The characteristic polynomial $\chi(M)$ is irreducible.
\item The characteristic polynomial $\chi(M)$ is not cyclotomic.
\item The characteristic polynomial $\chi(M)$ is not of the form $f(x^k),$ for some $k>1.$
\end{enumerate}
Then $\phi$ is pseudo-Anosov.
\end{theorem}
In view of Fact \ref{torellionto}, the question of showing that a generic element of $\mathcal{M}(S)$ is pseud-Anosov reduces to showing that a generic (in the sense of Section \ref{nonidealist}) element of the symplectic group $\Sp(2g, \integers)$ satisfies the conditions of Theorem \ref{cassonthm}. This is a done by what is, philosophically (in a sense that was later made precise by E. Kowalski in \cite{kowalski2008large}), a sieving argument. We show that the properties desired by Theorem \ref{cassonthm} are enjoyed with a probability \emph{independent of the prime $p$} by elements of the quotient group $\Sp(2g, \integers/p \integers).$ Since, by strong approximation (see \cite{MR1278263}, or, for a more elementary approach, in \cite{newman1972integral}),  the reductions modulo different primes are independendent, and the properties described in Theorem \ref{cassonthm} are assymetric, in the sense that, for example, in order to conclude that a polynomial is irreducible it is enough to find a \emph{single} prime for which the reduction mod $p$ is irreducible\footnote{ Note that the argument we used for $\SL(2, \integers)$ is much simpler, since there only one (large) prime sufficed, and no strong approximation argument was necessary -- it turns out that this kind of argument works for $\SL(n, \integers)$ to show that the characteristic polynomial of a generic matrix is irreducible, since it can be shown that the set of irreducible polynomials is a relatively sparse set of subvarieties of the set of coefficients}, brings us close to the end. The end is achieved thanks to the fundamental equidistribution result Theorem \ref{finitethm}, together with property $T$ for the groups we are studying for the groups we are interested in to assure exponential convergence (\cite{rivin2009walks}).
\subsection{The good news} The argument sketched above (see \cite{rivin2008walks,rivin2009walks} for all the details) has many virtues. Firstly, it is very general (just how general was outlined by Lubotzky and Meiri in \cite{lubotzky2012sieve}). 

In particular, it can be used to show that a generic element in the outer automorphism group of a free group is \emph{irreducible with irreducible powers}, which is the analogue in that setting of being pseudo-Anosov (an element $\psi$ of the automorphism group of a free group is \emph{irreducible} if it does not preserve any splitting of the free group $F$ as a free product $F=G \star H$; it is irreducible with irreducible powers (iwip) if all powers $\psi^k$ are, likewise, irreducible). The importance of iwip automorphisms was first noted in the foundational paper of M. Bestvina and M. Handel \cite{bestvina1992train}. For automorphisms of free groups there is the analogue of the Torelli homomorphism, which sends an automorphism $\psi$ of a free group $F_n$  to its action on the abelianization $\integers^n$ of $F_n.$ It is easy to show that the Torelli homomorphism $\mathfrak{T}$ is surjective onto the automorphism group of $\integers^n$ -- $\GL(n, \integers),$ and it is easy to see that an automorphism $\psi$ is irreducible if (of course, \emph{not} only if) the characteristic polynomial of $\mathfrak{T}(\psi)$ is irreducible. Initially, it seems a little frightening to check that $\psi$ is iwip by checking the characterstic polynomials of $(\mathfrak{T}(\psi))^k$ for \emph{every} $k$ for irreducibility, but it turns out (see \cite{rivin2008walks}) that it is enough to check that the Galois group of $\chi(\mathfrak{T}(\psi))$ is the full symmetric group. This can be proved by combining the previous ideas with the idea (going back to van der Waerden) of characterizing Galois groups \emph{via} the factorization patterns of polynomials modulo various primes. This result was later extended in \cite{jouve2010splitting} to show that characteristic polynomials of matrices in lattices in semisimple Lie group usually have as big a Galois group as possible (which is to say, the Weyl group of the ambient Lie group).

The results are effective (and explicit) n that they give an exponential rate of convergence of the densities to $0.$

The results can be extended without any work to finite index subgroups of $\SL(n, \integers)$ and $\Sp(2g, \integers),$ and their preimages in the mapping class groups, and, with some work,and much use of the results of \cite{breuillard2011approximate}, to \emph{thin Zariski dense subgroups} of such groups, and \emph{their} preimages (these results are still effective, but considerably less explicit than for lattices. It can also be shown, using the results of \cite{MR2725508} that for a \emph{generic} subgroup $H$ of the mapping class group, a generic element of $H$ is pseudo-Anosov.

However.
\subsection{Bad news}
In the context of mapping class groups, our results are not useful for groups which have very small image under the Torelli homomorphism. In particular, the Torelli group itself is completely ``orphaned'' -- Theorem \ref{cassonthm} is vacuous for elements in the Torelli group. This is all the more galling, since more geometric approaches (see Section \ref{mahersec}) show that a generic element of a subgroup of the mapping class group which contains at least two non-commuting pseudo-Anosov elements (below we will call such subgroups \emph{nonelementary}) is pseudo-Anosov. The problem with these approaches is that the convergence rates are completely ineffective, and also they do not apply in the less-geometric situations like the automorphism group of a free group.

\subsection{Better news}
Recently, at least some of the news became less bad, since two groups: A. Lubotzky and C. Meiri (\cite{lubotzky2011sieve}) and J. Malestein and J. Souto (\cite{malestein2011genericity} have extended the results sketched above to the Torelli group, using very similar methods. Lubotzky and Meiri have also extended their results to the Torelli subgroup of the automorphism group of the free group in \cite{lubotzkysieve}. (The results on the Torelli group make sense only when the genus of the surface in question is at least three -- genericity is hard to define for infinitely generated groups).

The first idea of these results goes to the beautiful paper of E. Looijenga \cite{looijenga1997prym}. later developed in a more algebraic direction by F. Grunewald and A. Lubotzky (\cite{grunewald2009linear}): 

Consider a surface $S$ and a double cover $\widetilde{S}.$ Any homeomorphism the Torelli group lifts to a homeomorphism of $\widetilde{S}.$ In addition, the image of the lift of the Torelli group under the Torelli homomorphism of $\widetilde{S}$ is of finite index in $P\Sp(2g-2, \integers),$ where $P\Sp(\bullet, \integers)$ denotes $\Sp(\bullet, \integers)/\{\pm I\}.$  

This gives us an indication that we might be able to use linear methods to study the Torelli group. The next ingredient goes back to the work of N. V. Ivanov \cite{MR1195787}: 
\begin{theorem}[N. V. Ivanov]
Any non-pseudo-Anosov element of the Torelli subgroup of a surface $S$ leaves invariant an essential  simple curve $\gamma$ on $S.$
\end{theorem}

Note that a much stronger result was shown by B. Farb, C. Leininger, and D. Margalit in \cite{farb2008lower}:
\begin{theorem}[\cite{farb2008lower}[Proposition 1.4]
Let $\gamma$ be a curve and $f$ an element in the Torelli subgroup. Then $i(\gamma, f(\gamma)) \geq 4$ if $\gamma$ is nonseparating, then $i(\gamma, f^j(\gamma)) \geq 2,$ for $j=1$ or $j=2,$
\end{theorem}
where $i(x, y)$ denotes the geometric intersection number of curves $x$ and $y.$

Finally, it is noted that $\gamma$ can be used to construct a cover such that the element $g$ in the corresponding$P\Sp(2g -2, \integers)$ leaves invariant a line in $\integers{2g-2},$ from which the genericity of pseudo-Anosovs in the \emph{whole} Torelli follows.

In fact, one can combine the above with the methods and results of \cite{MR2725508} it can be shown that a for a generic subgroup of Torelli, a generic element is pseudo-Anosov. However, the silver bullet would be the following:
\begin{conjecture}
\label{solvconj}
For any nonelementary subgroup $H$ of the mapping class group there is a cover $\widetilde{S}$ to which $H$ lifts, and such that the image of $H$ under the Torelli homomorphism is not solvable, with the degree of the cover at most polynomial in the sums of the wordlengths of the generators of $H.$
\end{conjecture}
It should be noted that we are \emph{very} far from being able to resolve Conjecture \ref{solvconj}. For example, while it is known that for every pseudo-Anosov mapping class $\psi$ there exists \emph{some} cover to which $\psi$ lifts, and such that the $\psi$ is not in the Torelli subgroup for that cover (\cite{koberda2012asymptotic}), it is \emph{not} known that the image $\mathfrak{T}(\psi)$ is of infinite order! The degree of the cover is effective -- the bounds in the paper follow essentially from the results of Edna Grossman (\cite{grossman1974residual} -- but, just as Grossman's paper, are easily doubly exponential in the word length of the element.

\section{The geometric approach}
\label{mahersec}
Above we have alluded to the ``geometric'' approach to the mapping class, which uses the curve complex. This approach was used by Joseph Maher in \cite{MR2772067} (where he also shows other remarkable results). Maher's results are very general, but not effective. A somewhat more conceptual approach was undertaken in the very nice paper by A. Malyutin \cite{malyutin2011quasimorphisms}. Malyutin's approach is as follows:

First, he uses the central limit theorem of M. Bj\"orklund and T. Hartnick\cite{bjorklund2010biharmonic} (which is a vast generalization, at the cost of losing effectiveness completely of some of the results of the author's paper \cite{rivin1999growth}) to show the following:
\begin{theorem}[A. Malyutin]
\label{malythm}
Let $G$ be a countable group and let $\Phi: G \mapsto \reals^d$ is a nondegenerate $\reals^d$-quasimorphism. Then, for each nondegenerate probability measure $\mu$ on $G$ and for every bounded subset $Q \subset \reals^d,$ there is a constant $C = C(G, \Phi, \mu, Q)$ such that for any $k \in \mathbb{N}$ and $\mathbf{x} \in \reals^d$ we have 
\[ \mu^{\ast k}
\left(\Phi^{-1}(\mathbf{x} + Q)\right) < C k^{-d/2},
\] 
where $\mu^{\ast k}$ denotes the $k$-fold convolution of $\mu$ with itself.
\end{theorem}

Recall that a \emph{quasimorphism} from a group $G$ to $\reals$ is a map $\phi: G \mapsto\reals$ such that the following condition holds:
\begin{equation}
\label{quasieq}
\sum_{h, g \in G} \left|\phi(gh) - \phi(g) - \phi(h)\right| < \infty.
\end{equation}
An $\reals^d$ quasimorphism is a map $\Phi: G \rightarrow \reals^d$ which satisfies the inequality \eqref{quasieq} (where $\left|\bullet\right|$ now denotes some Banach norm) -- a map is an $\reals^d$ quasimorphism if and only if its coordinates are garden-variety quasimorphisms. Such a $\Phi$ is called \emph{nondegenerate} if its image is not contained in a proper hyperplane.

An immediate corollary of Theorem \ref{malythm} is the following Corollary:
\begin{corollary}
\label{malycorr}
If a subset $S$ of a countable group $G$ has bounded image under a nondegenerate $\reals^d$-quasimorphism $G \mapsto \reals^d,$ then for every nondegenerate probability measure $\mu$ on $G$ there exists a constant $C = C(\mu),$ such that for each $k \in \mathbb{N}$ we have 
\[
\mu^{\ast k}(S) < Ck^{-d/2}.
\]
\end{corollary}

The other ingredient is the result of M. Bestvina and K. Fujiwara \cite{bestvina2002bounded}, which states that if $H$ is a non-elementary subgroup of the mapping class group of a surface, than there are infinitely many linearly independent quasimorphisms which all map the non-pseudo-anosov mapping classes to $0.$ 

The Bestvina-Fujiwara Theorem and Corollary \ref{malycorr} together show that the probability of being non-pseudo-Anosov decreases faster than any polynomial (but does not quite show exponential decay. All constants in the argument are completely ineffective).
\bibliographystyle{plain}
\bibliography{msri}
\end{document}